\documentclass[reqno,12pt,a4paper]{amsart}

\usepackage{tikz}

\usepackage[mathscr]{eucal}
\usepackage{graphics,epic}
\usepackage{amsfonts}
\usepackage{amscd}
\usepackage{latexsym}
\usepackage{amsmath,amssymb, amsthm, stmaryrd, bm, bbm}
\usepackage[all,2cell]{xy}

\newtheorem{theorem}{Theorem}[section]
\newtheorem*{theorem*}{Theorem}

\newtheorem{lemma}[theorem]{Lemma}
\newtheorem{proposition}[theorem]{Proposition}
\newtheorem{corollary}[theorem]{Corollary}

\newtheorem*{conjecture*}{Conjecture}

\newtheorem*{question*}{Question}
\theoremstyle{remark}

\theoremstyle{definition}
\newtheorem{definition}[theorem]{Definition}

\newcommand{\dimv}{\underline{\dim}\,}

\newcommand{\ch}{{\mathcal H}}

\newcommand{\cp}{{\mathcal P}}

\numberwithin{equation}{section}

\begin{document}

\title[Ringel--Hall algebra of $\Lambda(n-1,1,1)$]{On the Ringel--Hall algebra of the gentle one-cycle algebra $\Lambda(n-1,1,1)$}

\author{Hui Chen, and Dong Yang}

\address{Hui Chen, School of Biomedical Engineering and Informatics, Nanjing Medical University, Nanjing 211166, P. R. China.}
\email{huichen@njmu.edu.cn}

\address{Dong Yang, Department of Mathematics, Nanjing University, Nanjing 210093, P. R. China}
\email{yangdong@nju.edu.cn}

\begin{abstract} It is shown that the gentle one-cycle algebra $\Lambda(n-1,1,1)$ has Hall polynomials. The Hall polynomials are explicitly given for all triples of indecomposable modules, and as a consequence, the Ringel--Hall Lie algebra of $\Lambda(n-1,1,1)$ is shown to be isomorphic to its Riedtmann Lie algebra.
\end{abstract}
\maketitle

\section{Introduction}
\label{makereference:intro}

Let $A$ be a finite-dimensional algebra over a finite field $\mathbb{F}_q$. The \emph{Ringel--Hall algebra} $\ch(A)$ of $A$ is the $\mathbb{Z}$-algebra with basis the set of isomorphism classes of finite-dimensional left $A$-modules and the corresponding structure constants given by filtration numbers. Assume that $A$ is representation-finite, i.e., there are only finitely many isomorphism classes of finite-dimensional $A$-modules. Then $A$ is said to have Hall polynomials if, roughly speaking, the filtration numbers are the evaluation at $q$ of some integral polynomials. In this case, we can view $q$ as an indeterminate and evaluate $\mathcal{H}(A)$ at $q=1$. Then the $\mathbb{Z}$-submodule $\mathfrak{g}(A)$ of $\ch(A)$ spanned by the isoclasses of indecomposable $A$-modules is a Lie subalgebra. In this way, the positive parts of finite-dimensional simple complex Lie algebras can be constructed from the representation theory of Dynkin quivers. See \cite{Ringel88,Ringel90,Ringel92}.

It is proved in \cite{Ringel88} that Hall polynomials exist for representation-finite algebras with direct Auslander--Reiten quiver and in general this is an open problem. In this paper we consider the gentle one-cycle algebra $A$ given by the quiver with relation
\begin{equation*}
\xymatrix@R=1ex{
1 \ar[r]&2\ar[r]&\cdots\ar[r]&n\ar@(ur,dr)^{\alpha}},\qquad \alpha^2,
\end{equation*}
whose Auslander--Reiten quiver is not directed. 
Over an algebraically closed field this is the derived-discrete algebra $\Lambda(n-1,1,1)$ in the notation of \cite{BobinskiGeissSkowronski04}.
Our first main result is

\begin{theorem}
[{Theorem~\ref{Hall-exist-for-A}}]
$A$ has Hall polynomials.
\end{theorem}

Thus there is the Lie subalgebra $\mathfrak{g}(A)$ of $\ch(A)$ (evaluated at $q=1$) of indecomposable $A$-modules. On the other hand, following Riedtmann's approach \cite{Riedtmann94} to consider the representation theory of $A$ over $\mathbb{C}$, we can obtain another Lie algebra $L(A_{\mathbb{C}})$, see \cite{ChenYang24}. Our second main result is

\begin{theorem}[{Theorem~\ref{thm:iso-between-g(A)-and-L(A)}}]
$\mathfrak{g}(A)$ and $L(A_{\mathbb{C}})$ are isomorphic Lie algebras.
\end{theorem}

This is done by brutal force, namely, we compute the Hall polynomials for all triples of indecomposable $A$-modules (Proposition~\ref{Hallpoly}), and as a consequence we obtain the structure constants of $\mathfrak{g}(A)$ (Corollary~\ref{coro:lie-brackets-of-g(A)}), which are the same as those of $L(A_{\mathbb{C}})$ given in \cite{ChenYang24}.

\subsection*{Acknowledgement} The authors are very grateful to Bangming Deng and Haicheng Zhang for answering their questions. The second author acknowledges support by the National Natural Science Foundation of China No.12031007.

\section{Preliminaries}

In this section we recall some basic results on Ringel--Hall algebras and on the classification of indecomposable modules over the algebra $\Lambda(n-1,1,1)$.
\subsection{Ringel--Hall algebras}
\label{s:RHA}

Let $K=\mathbb{F}_q$ be a finite field, and $A$ be an associative unital finite-dimensional representation-finite $K$-algebra. By an $A$-module we will mean a finite-dimensional left $A$-module.  For an $A$-module $M$ we denote by $[M]$ its isomorphism class. 
We denote by $\mathcal{P}(A)$ the set of all isomorphism classes of $A$-modules. 

For three $A$-modules $N_1,N_2$ and $M$, let $F_{N_1N_2}^M$ be the \emph{Hall number}, i.e., the number of submodule $U$ of $M$ such that $U\cong N_2$ and $M/U\cong N_1$.
Let $\mathcal{H}(A)$ be the free $\mathbb{Z}$-module with basis $\mathcal{P}(A)$, and for any $[N_1], [N_2], [M]\in \mathcal{P}(A)$ define
\[[N_1]\diamond [N_2]=\sum_{[M]\in\mathcal{P}(A)}F_{N_1N_2}^M [M].\]

\begin{proposition}[\cite{Ringel90}]
$\mathcal{H}(A)$ is an associative $\mathbb{Z}$-algebra with unit $[0]$, the isoclass of the zero module.
\end{proposition}

The algebra $\ch(A)$ is called the \emph{Ringel--Hall algebra} of $A$. The following lemma is an immediate consequence of \cite[Lemma 3.5]{KasjanKosakowska22} (see also \cite[Lemma 2.1]{Nasr12}).

\begin{lemma}\label{compo}
Let $X, X_1, X_2, X_3, X_4, Y, M$ be $A$-modules (possibly 0), $a, b\in \mathbb{C}$ and assume that $[X]=a[X_1]\diamond [X_2]+b[X_3]\diamond [X_4]$ in $\mathcal{H}(A)$, then
\[
F_{XY}^M=a\sum_{Z\in\mathcal{P}(A)}F_{X_1 Z}^MF_{X_2 Y}^Z+b\sum_{Z\in\mathcal{P}(A)}F_{X_3 Z}^M F_{X_4 Y}^Z.
\]
\end{lemma}

Given an extension field $E$ of $K$ and a $K$-space $V$, put $V^E:=V\otimes_K E$. According to \cite{Ringel92}, a field extension $E$ of $K$ is said to be \emph{conservative} for the $K$-algebra $A$ if the algebra (End$M$/rad End $M$)$^E$ is a field for any indecomposable $A$-module $M$.

\begin{definition}[{\cite[Section 3]{Ringel92}}]
We say that $A$ \emph{has Hall polynomials} provided that for all $A$-modules $X, Y, M$, there exists a polynomial $\varphi_{XY}^M\in \mathbb{Z}[T]$ such that for any conservative field extension $E$ of $K$, we have
\[\varphi_{XY}^M(|E|)=F_{{X^E}{Y^E}}^{M^E}.\]
\noindent Such a $\varphi_{XY}^M$ is called a \emph{Hall polynomial}.
\end{definition}

For the existence of Hall polynomials, there are the following two general results:
\begin{lemma}\label{simp}
Let $S$ be a simple $A$-module. Then  the Hall polynomials $\varphi_{SN}^M$ and $\varphi_{NS}^M$ exist for any $A$-modules $M$ and $N$.
\end{lemma}

\begin{lemma}\label{proj}
Let $P$ be a projective $A$-module. Then the Hall polynomial $\varphi_{PN}^M$ exists  for any $A$-modules $M$ and $N$.
\end{lemma}
\begin{proof}
 Since $P$ is a projective $A$-module, the short exact sequence
\[0\rightarrow N\rightarrow M\rightarrow P\rightarrow0\]
is always split. By \cite[Lemma 2.7]{Nasr12}, the Hall polynomial $\varphi_{PN}^M$ exists.
\end{proof}
\begin{theorem}
[{\cite[Theorem 2.9]{Nasr12}}]\label{Hall-exist}
Assume that $A$ is representation-finite. Then  $A$ has Hall polynomials if and only if  the Hall polynomial $\varphi_{XY}^M$ exists for any indecomposable $A$-module $X$ and any $A$-modules $Y$ and $M$.
\end{theorem}

If $A$ has Hall polynomials, then the \emph{degenerate Ringel--Hall algebra} $\mathcal{H}(A)_1$ is defined as the $\mathbb{Z}$-algebra with basis $\cp(A)$ and with multiplication
\[[N_1][N_2]=\sum_{[M]}\varphi_{N_1N_2}^M(1) [M].\]
We denote by $\mathfrak{g}(A)$ the $\mathbb{Z}$-submodule of $\mathcal{H}(A)_1$ with basis $[M]$, where $M$ runs over all indecomposable $A$-modules.

The following theorem is part of \cite[Theorem]{Ringel92},
\begin{theorem}\label{Lie-alg}
$\mathfrak{g}(A)$ is a Lie subalgebra of $\mathcal{H}(A)_1$, and $\mathcal{H}(A)_1\otimes_{\mathbb{Z}}\mathbb{Q}$ is isomorphic as a $\mathbb{Q}-$algebra to the universal enveloping algebra of $\mathfrak{g}(A)\otimes_{\mathbb{Z}}\mathbb{Q}$.
\end{theorem}

\subsection{Representations of the algebra $\Lambda(n-1,1,1)$}
\label{s:R-of-GOA}

Let $K$ be an arbitrary field. 
Consider the bound quiver $(Q,I)$
\begin{equation*}
\xymatrix@R=1ex{
1 \ar[r]&2\ar[r]&\cdots\ar[r]&n\ar@(ur,dr)^{\alpha}},\qquad \alpha^2,
\end{equation*}
and denote the path algebra of this bound quiver by $A_K$. When $K$ is an algebraically closed field, this is the derived-discrete algebra $\Lambda(n-1,1,1)$ in the notation of \cite{BobinskiGeissSkowronski04}. We will identity an $A_K$-module with a representation of $(Q,I)$, i.e., a representation of $Q$ satisfying the relation $\alpha^2=0$. The following theorem was established as \cite[Theorem 3.3]{BoosReineke12} for $K=\mathbb{C}$, but it holds true for any field $K$ by for example \cite[Section 3, Theorem]{ButlerRingel87}.

\begin{theorem}\label{repn}
The following representations $U_{i,j}$, $V_i$ and $W_{i,j}$ form a complete set of representatives of the isoclasses of indecomposable $A_K$-modules:
\begin{itemize}
  \item $U_{i,j}$ for $1 \leq j\leq i \leq n$:\\
\[\xymatrix@R=1ex{
0 \ar[r]^0&\cdots\ar[r]^0&0 \ar[r]^0&K\ar[r]^1&\cdots\ar[r]^1&K \ar[r]^{e_1}&K^2\ar[r]^{\rm I}&\ldots\ar[r]^{\rm I}&K^2\ar@(ur,dr)^{M_\alpha}\\
&&&\bullet\raisebox{-4mm}{\hspace{-2mm}$j$}\ar[r]&\cdots\ar[r]&\bullet \ar[r]&\bullet\ar[r] \raisebox{-4mm}{\hspace{-2mm}$i$}&\cdots\ar[r]&\bullet\ar@(r,r)[d]\raisebox{-4mm}{\hspace{-2mm}$n$}\\
&&&&&&\bullet\ar[r]&\cdots\ar[r]&\bullet}\\\]
  \item $U_{i,j}$ for $1 \leq i < j \leq n$:\\
\[\xymatrix@R=1ex{
0 \ar[r]^0&\cdots\ar[r]^0&0 \ar[r]^0&K\ar[r]^1&\ldots\ar[r]^1&K \ar[r]^{e_2}&K^2\ar[r]^{\rm I}&\ldots\ar[r]^{\rm I}&K^2\ar@(ur,dr)^{M_\alpha}\\
&&&&&&\bullet\ar[r]&\cdots\ar[r]&\bullet\ar@(r,r)[d]\\
&&&\bullet\raisebox{4mm}{\hspace{-2mm}$i$}\ar[r]&\ldots\ar[r]&\bullet \ar[r]&\bullet\ar[r] \raisebox{4mm}{\hspace{-2mm}$j$}&\cdots\ar[r]&\bullet\raisebox{4mm}{\hspace{-2mm}$n$}}\\\]
  \item $V_i$ for $1 \leq i \leq n$:\\
\[\xymatrix@R=1ex{
0 \ar[r]^0&\cdots\ar[r]^0&0 \ar[r]^0&K\ar[r]^1&\cdots\ar[r]^1&K\ar@(ur,dr)^{0}\\
&&&\bullet\raisebox{4mm}{\hspace{-2mm}$i$}\ar[r]&\cdots\ar[r]&\bullet\raisebox{4mm}{\hspace{-2mm}$n$}}\\\]
  \item $W_{i,j}$ for $1 \leq i \leq j <n$:\\
\[\xymatrix@R=1ex{
0 \ar[r]^0&\cdots\ar[r]^0&0 \ar[r]^0&K\ar[r]^1&\cdots\ar[r]^1&K\ar[r]^{0}&0\ar[r]^{0}&\cdots\ar[r]^{0}&0\ar@(ur,dr)^{0}\\
&&&\bullet\raisebox{4mm}{\hspace{-2mm}$i$}\ar[r]&\cdots\ar[r]&\bullet\raisebox{4mm}{\hspace{-2mm}$j$}}.\\\]
\end{itemize}
\noindent Here, $e_1 =\left(
              \begin{array}{c}
                1 \\
                0 \\
              \end{array}
            \right)$
, $e_2 =\left(
              \begin{array}{c}
                0 \\
                1 \\
              \end{array}
            \right)$, $\rm{I}=\left(
                   \begin{array}{cc}
                     1 & 0 \\
                     0 & 1 \\
                   \end{array}
                 \right)$, and $M_\alpha=\left(
              \begin{array}{cc}
                0 & 0 \\
                1 & 0 \\
              \end{array}
            \right)
$.
\end{theorem}

In particular, $A_K$ is representation-finite. The simple $A_K$-modules are $S_1=W_{1,1},\ldots,S_{n-1}=W_{n-1,n-1}$ and $S_n=V_{n}$. Moreover, $P_1=U_{n,1},\ldots, P_n=U_{n,n}$ (respectively, $I_1=W_{1,1},\ldots, I_{n-1}=W_{1,n-1}, I_n=U_{1,n}$) form a complete set of pairwise non-isomorphic indecomposible projective (respectively, injective) $A_K$-modules.\\

\section{The Ringel--Hall algebra of $\Lambda(n-1,1,1)$}

In this section, we study the Ringel--Hall algebra of the algebra $\Lambda(n-1,1,1)$. More precisely, we show the existence of Hall polynomials, compute the Hall polynomials $\varphi_{XY}^M$ for all triples $(X,Y,M)$ of indecomposable modules and describe the structure constants of the Ringel--Hall Lie algebra.

\subsection{The existence of Hall polynomials}
\label{s:existence-Hall-polynomial}

Fix a finite field $\mathbb{F}_q$ and put $A=A_{\mathbb{F}_q}$ (see Section~\ref{s:R-of-GOA}). The main result of this subsection is

\begin{theorem}\label{Hall-exist-for-A}
The algebra $A$ has Hall polynomials.
\end{theorem}
\begin{proof}
By Theorem \ref{Hall-exist}, we only need to prove that for any indecomposable $A$-module $X$ the Hall polynomial $\varphi_{XY}^M$ exists for any $A$-modules $Y$ and $M$.

Case 1: $X=W_{i,j}$, $1\leq i\leq j\leq n-1$. For $j=i$, this follows by Lemma~\ref{simp}, since $W_{i,i}=S_i$ is a simple module.
If $j>i$, we have 
\[
[W_{i,i}]\diamond [W_{i+1,j}]=[W_{i,i}\oplus W_{i+1,j}]+[W_{i,j}],\hbox{ and } [W_{i+1,j}]\diamond [W_{i,i}]=[W_{i,i}\oplus W_{i+1,j}],
\] 
so 
\[
[W_{i,j}]=[W_{i,i}]\diamond [W_{i+1,j}]-[W_{i+1,j}]\diamond [W_{i,i}].\] 
Therefore it follows  by Lemma~\ref{compo} that for any $A$-modules $Y$ and $M$,,
\[F_{W_{i,j}Y}^M=\sum_{Z\in \mathcal{P}(A)}F_{W_{i,i}Z}^MF_{W_{i+1,j}Y}^Z-\sum_{Z\in \mathcal{P}(A)}F_{W_{i+1,j}Z}^MF_{W_{i,i}Y}^Z.\]
Hence by induction the Hall polynomial $\varphi_{XY}^M$ exists for $X=W_{i,j}$, $1\leq i\leq j\leq n-1$.

Case 2: $X=V_i$, $1\leq i\leq n$. For $i=n$ this follows by Lemma~\ref{simp}, since $V_n=S_n$ is a simple module. For $1\leq i\leq n-1$, we have 
\[
[W_{i,n-1}]\diamond [S_{n}]=[W_{i,n-1}\oplus S_{n}]+[V_{i}],\hbox{ and }[S_{n}]\diamond [W_{i,n-1}]=[W_{i,n-1}\oplus S_{n}],
\] 
so 
\[
[V_{i}]=[W_{i,n-1}]\diamond [S_{n}]-[S_{n}]\diamond [W_{i,n-1}].
\]
Therefore it follows by Lemma~\ref{compo} that for any $A$-modules $Y$ and $M$,
\[F_{V_{i}Y}^M=\sum_{Z\in \mathcal{P}(A)}F_{W_{i,n-1}Z}^MF_{S_{n}Y}^Z-\sum_{Z\in \mathcal{P}(A)}F_{S_{n}Z}^MF_{W_{i,n-1}Y}^Z.\]
Hence by Case 1 the Hall polynomial $\varphi_{XY}^M$ exists for $X=V_{i}$, $1\leq i\leq n$.

Case 3: $X=U_{i,j}$, $1\leq j\leq i\leq n$.
For $i=n$ this follows by Lemma~\ref{proj}, since $U_{n,j}=P_j$ is a projective module. For $1\leq j\leq i\leq n-1$, we have 
\[
[W_{i,n-1}]\diamond [P_{j}]=q[W_{i,n-1}\oplus P_{j}]+q[U_{i,j}],\hbox{ and }[P_{j}]\diamond [W_{i,n-1}]=[W_{i,n-1}\oplus P_{j}],
\] 
so 
\[
[U_{i,j}]=q^{-1}[W_{i,n-1}]\diamond [P_{j}]-[P_{j}]\diamond [W_{i,n-1}].
\]
Therefore it follows by Lemma~\ref{compo} that for any $A$-modules $Y$ and $M$,
\[F_{U_{i,j}Y}^M=q^{-1}\sum_{Z\in \mathcal{P}(A)}F_{W_{i,n-1}Z}^MF_{P_{j}Y}^Z-\sum_{Z\in \mathcal{P}(A)}F_{P_{j}Z}^MF_{W_{i,n-1}Y}^Z.\]
Hence by Case 2, there exists a rational function $\varphi_{XY}^M$ such that $F_{U_{i,j}Y}^M=\varphi_{XY}^M(q)$. Since  this is an integer for any $q$, it follows by \cite[Section 2, Lemma]{Ringel88} that $\varphi_{XY}^M$ is a polynomial. Namely, the Hall polynomial $\varphi_{XY}^M$ exists for $X=U_{i,j}$, $1\leq j\leq i\leq n$.

Case 4: $X=U_{i,j}$, $1\leq i<j\leq n$.
We have
\[
[V_{i}]\diamond [V_{j}]=q[V_{i}\oplus V_{j}]+[U_{j,i}],\hbox{ and }[V_{j}]\diamond [V_{i}]=[V_{i}\oplus V_{j}]+[U_{i,j}],
\] 
so 
\[
[U_{i,j}]=[V_{j}]\diamond [V_{i}]+q^{-1}[U_{j,i}]-q^{-1}[V_{i}]\diamond [V_{j}].
\]
Therefore it follows by Lemma~\ref{compo} that for any $A$-modules $Y$ and $M$,
\[F_{U_{i,j}Y}^M=\sum_{Z\in \mathcal{P}(A)}F_{V_{j}Z}^MF_{V_{i}Y}^Z+q^{-1}F_{U_{j,i}Y}^M-q^{-1}\sum_{Z\in \mathcal{P}(A)}F_{V_{i}Z}^MF_{V_{j}Y}^Z.\]
Hence by Case 2 and Case 3, there exists a rational function $\varphi_{XY}^M$ such that $F_{U_{i,j}Y}^M=\varphi_{XY}^M(q)$. Then we can show that the Hall polynomial $\varphi_{XY}^M$ exists as in Case 3.
\end{proof}

\subsection{Hall polynomials for indecomposable modules}
In this subsection we explicitly give the Hall polynomial $\varphi_{XY}^M$ for indecomposable $A$-modules $X,Y,M$.

\label{Hallpoly-for-A}
\begin{proposition}\label{Hallpoly}
Let $X,Y,M$ be indecomposable $A$-modules. Then the Hall polynomial $\varphi_{XY}^M$ is $0$ except in the following cases:

\begin{itemize}
  \item $\varphi_{XY}^M=T$ if the triple $(X,Y,Z)$ is one of the following:
  \begin{itemize}
           \item [\rm (1)] $(W_{i,j},U_{j+1,l},U_{i,l})$ for $i\leq l\leq j<n$,
           \item [\rm (2)] $(W_{i,j},U_{j+1,j},U_{i,j})$ for $i<j<n$.
         \end{itemize}
  \item $\varphi_{XY}^M=1$ if the triple $(X,Y,Z)$ is one of the following:
  \begin{itemize}
          \item [\rm (3)] $(W_{i,j},W_{j+1,l},W_{i,l})$ for $i\leq j\leq l<n$,
          \item [\rm (4)] $(W_{i,j},V_{j+1},V_{i})$ for $i\leq j<n$,
          \item [\rm (5)] $(W_{i,j},U_{j+1,l},U_{i,l})$ for $i\leq j<j+1<l\leq n$,
          \item [\rm (6)] $(W_{i,j},U_{l,j+1},U_{l,i})$ for $i\leq j<j+1\leq l\leq n$,
          \item [\rm (7)] $(W_{i,j},U_{l,j+1,l},U_{l,i})$ for $l<i\leq j<n$,
          \item [\rm (8)] $(W_{i,j},U_{l,j+1},U_{l,i})$ for $i<l\leq j<n$,
          \item [\rm (9)] $(W_{i,j},U_{j+1,l},U_{i,l})$ for $i\leq j<n$,
          \item [\rm (10)] $(W_{i,j},U_{i,j+1},U_{i,i})$ for $i\leq j<n$,
          \item [\rm (11)] $(V_{i},V_{j},U_{i,j})$ for $1\leq i,j\leq n$.
        \end{itemize}
\end{itemize}
\end{proposition}
\begin{proof}
We only prove (1),(5),(11), and the proof of the other cases is similar.

Case (1): For $i\leq l\leq j<n$, it is clear that $\dimv W_{i,j}+\dimv U_{j+1,l}=\dimv U_{i,l}$, and from the structure of $W_{i,j}$, $U_{j+1,l}$ and  $U_{i,l}$ in Section~\ref{s:R-of-GOA}, $U_{i,l}$ has $q$ submodules isomorphic to $U_{j+1,l}$ with quotient isomorphic to $W_{i,j}$, which are:
\[\xymatrix@C=1.4pc@R=1ex{
0 \ar[r]^(0.4)0&\cdots\ar[r]^0&0 \ar[r]^(0.25)0&K\{e_1+ce_2\}\ar[r]^(0.65)1&\ldots\ar[r]^(0.3)1&K\{e_1+ce_2\} \ar@{^(->}[r]&K^2\ar[r]^{\rm I}&\ldots\ar[r]^{\rm I}&K^2\ar@(ur,dr)^{M_\alpha}\\
&&&\raisebox{-1mm}{\hspace{-2mm}$l$}&&&\raisebox{-1mm}{\hspace{-2mm}$j+1$}&&\raisebox{-1mm}{\hspace{-2mm}$n$}\\
&&&&&&&&}\\\]
where $c\in K=\mathbb{F}_q$.

Case (5): For $i\leq j<j+1<l\leq n$, it is clear that $\dimv W_{i,j}+\dimv U_{j+1,l}=\dimv U_{i,l}$, and from the structure of $W_{i,j}$, $U_{j+1,l}$ and  $U_{i,l}$ in Section~\ref{s:R-of-GOA}, $U_{i,l}$ has a unique submodule isomorphic to $U_{j+1,l}$ with quotient isomorphic to $W_{i,j}$.

%

Case (11): For $1\leq i,j\leq n$, it is clear that $\dimv V_{i}+\dimv V_{j}=\dimv U_{i,j}$, and from the structure of $V_i$, $V_j$ and  $U_{i,j}$ in Section~\ref{s:R-of-GOA}, $U_{i,j}$ has a unique submodule isomorphic to $V_{j}$ with quotient isomorphic to $V_{i}$.
\end{proof}
\subsection{The Ringel--Hall Lie algebra}
We have shown in Theorem~\ref{Hall-exist-for-A} that the algebra $A$ has Hall polynomials. As in the end of Section~\ref{s:RHA}, let $\mathfrak{g}(A)$ be the $\mathbb{Z}$-submodule of the degenerate Ringel--Hall algebra $\mathcal{H}(A)_1$ with basis $[M]$, where $M$ runs over all indecomposable $A$-modules. The following is a corollary of Theorem \ref{Lie-alg} and Proposition \ref{Hallpoly}:

\begin{corollary}
\label{coro:lie-brackets-of-g(A)}
$\mathfrak{g}(A)$ is a Lie subalgebra of $\mathcal{H}(A)_1$ with the following Lie bracket:
\begin{itemize}
         \item[\rm (a)] $[W_{i,j}, W_{l,m}]=\delta_{j+1,l}W_{i,m}-\delta_{m+1,i}W_{l,j}$ for $1\leq i\leq j\leq n-1$ and $1\leq l\leq m\leq n-1$,
         \item[\rm (b)] $[W_{i,j}, V_l]=\delta_{j+1,l}V_i$ for $1\leq i\leq j\leq n-1$ and $1\leq l\leq n$,
         \item[\rm (c)] $[W_{i,j}, U_{l,m}]=\delta_{j+1,m}U_{l,i}+\delta_{j+1,l}U_{i,m}$ for $1\leq i\leq j\leq n$ and $1\leq l,m\leq n$,
         \item[\rm (d)] $[V_i, V_j]=U_{j,i}-U_{i,j}$ for $1\leq i,j\leq n$,
\end{itemize}
\noindent where $\delta_{i,j}$ is the Kronecker symbol. For other pairs of indecomposable modules, the Lie bracket of them equals 0.
\end{corollary}

Using Riedtmann's construction \cite{Riedtmann94}, we have obtained in \cite{ChenYang24} a Lie algebra $L(A_{\mathbb{C}})$ from the representation theory of $A_{\mathbb{C}}$, which is isomorphic to a Lie algebra of type $\mathbf {BC}^+$. By Corollary \ref{coro:lie-brackets-of-g(A)} and \cite[Proposition 4.2]{ChenYang24}, we have
\begin{theorem}
\label{thm:iso-between-g(A)-and-L(A)}
There is an isomorphism $\mathfrak{g}(A)\cong L(A_{\mathbb{C}})$ of Lie algebras.
\end{theorem}


\begin{thebibliography}{10}

\bibitem{BobinskiGeissSkowronski04}
Grzegorz Bobi{\'n}ski, Christof Gei{\ss}, and Andrzej Skowro{\'n}ski,
  \emph{Classification of discrete derived categories}, Cent. Eur. J. Math.
  \textbf{2} (2004), no.~1, 19--49 (electronic).

\bibitem{BoosReineke12}
Magdalena Boos and Markus Reineke, \emph{{$B$}-orbits of 2-nilpotent matrices
  and generalizations}, Highlights in {L}ie algebraic methods, Progr. Math.,
  vol. 295, Birkh\"{a}user/Springer, New York, 2012, pp.~147--166.

\bibitem{ButlerRingel87}
M.~C.~R. Butler and Claus~Michael Ringel, \emph{Auslander-{R}eiten sequences
  with few middle terms and applications to string algebras}, Comm. Algebra
  \textbf{15} (1987), no.~1-2, 145--179. \MR{876976 (88a:16055)}

\bibitem{ChenYang24}
Hui Chen and Dong Yang, \emph{On {R}iedtmann's {L}ie algebra of the gentle
  one-cycle algebra ?(n-1,1,1)}, Communications in Algebra \textbf{52} (2024),
  no.~1, 345--358.

\bibitem{KasjanKosakowska22}
Stanislaw Kasjan and Justyna Kosakowska, \emph{The existence of {H}all
  polynomials for {$x^2$}-bounded invariant subspaces of nilpotent linear
  operators}, J. Pure Appl. Algebra \textbf{226} (2022), no.~5, Paper No.
  106921, 12.

\bibitem{Nasr12}
A.~R. Nasr-Isfahani, \emph{Hall polynomials for {N}akayama algebras}, Algebr.
  Represent. Theory \textbf{15} (2012), no.~3, 483--490.

\bibitem{Riedtmann94}
Christine Riedtmann, \emph{Lie algebras generated by indecomposables}, J.
  Algebra \textbf{170} (1994), no.~2, 526--546.

\bibitem{Ringel88}
Claus~Michael Ringel, \emph{Hall algebras}, Topics in algebra, Part 1 (Warsaw,
  1988), Banach Center Publ., vol.~26, PWN, Warsaw, 1990, pp.~433--447.

\bibitem{Ringel90}
\bysame, \emph{Hall algebras and quantum groups}, Invent. Math. \textbf{101}
  (1990), no.~3, 583--591.

\bibitem{Ringel92}
Claus~Michael Ringel, \emph{Lie algebras arising in representation theory},
  Representations of algebras and related topics ({K}yoto, 1990), London Math.
  Soc. Lecture Note Ser., vol. 168, Cambridge Univ. Press, Cambridge, 1992,
  pp.~284--291.

\end{thebibliography}

\def\cprime{$'$}
\providecommand{\bysame}{\leavevmode\hbox to3em{\hrulefill}\thinspace}
\providecommand{\MR}{\relax\ifhmode\unskip\space\fi MR }
\providecommand{\MRhref}[2]{%
  \href{http://www.ams.org/mathscinet-getitem?mr=#1}{#2}
}
\providecommand{\href}[2]{#2}

\end{document}